\documentclass{article}
\usepackage {amssymb}
\usepackage {amsmath}
\usepackage {amsthm}
\usepackage{graphicx}
\usepackage {amscd}
\usepackage[colorlinks, citecolor=blue]{hyperref}
\input{diagrams.sty}
\newcommand{\DR}{{\mathcal{DR}}}
\newcommand{\D}{{\mathcal{D}}}

\newcommand{\mld}{{\rm mld}}
\newcommand{\loc}{{\rm loc}}
\newcommand{\orb}{{\rm orb}}

\newtheorem{thm}{Theorem}
\newtheorem{prop}{Proposition}

\newtheorem{rem}{Remark}
\newtheorem{conj}{Conjecture}

\newtheorem{lem}{Lemma}

\input{xy}
\xyoption{all}

\begin{document}
\title{Finiteness of algebraic fundamental groups}
\author{Chenyang Xu}
\date{\today}
\maketitle{}
\begin{abstract} We show that the algebraic local fundamental group of any klt singularity as well as the algebraic fundamental group of the smooth locus of any log Fano variety are finite. 
\end{abstract}

\tableofcontents

\bigskip


\vspace{4mm}

\noindent {\bf MSC class}: 14J17, 14J45.

\section{Introduction}

We work over the field $\mathbb{C}$ of complex numbers. The study of the local topology of singularities has a long history. In the surface case, Mumford proved that a point in a normal surface has a trivial local fundamental group if and only if it is smooth (cf. \cite{Mumford61}).  Since then the investigation of the local topology of a singularity has been one of the most important tools to study singularities. It has been understood that there are three basic objects to study. Given a singularity, the {\it link} $L(0\in X)$ should carry essentially all the local topological information of a singularity. It has a continuous map to a topological space whose deformation retract is the {\it simple normal crossing variety} $E$ defined as the preimage of 0 for a log resolution $Y\to (X,0)$. The combinatorial gluing data of $E$ is then captured in the dual complex $\D(E)=\DR(0\in X)$.  See \cite{Kollar12} for more background.

Recently,  examples (cf. \cite{Kollar11b, KK11, Kollar12}) have been constructed to show that for a general singularity, the dual complex can be as complicated as possible. 
When $0\in X$ is log canonical, Koll\'ar also constructed 3-dimensional examples that have more complicated local topology than people expected. For instance the local fundamental group of such a singularity can be the fundamental group of any connected two dimensional manifold.
Koll\'ar indeed asked whether there is any nontrivial restriction of $\pi_1(E)$ (cf. \cite[25]{Kollar11b}).

In this note, we aim to show that the local  topology of a klt singularity should be much simpler than the log canonical case. In fact, the following conjecture is proposed by Koll\'ar.
\begin{conj}[\cite{Kollar11b}, 26] Let $0\in (X,\Delta)$ be a klt singularity. Then the local fundamental group $\pi^{\loc}_1(X,0):=\pi_1(L(0\in X))$ is finite. 
\end{conj}
In this direction, Koll\'ar and Takayama proved that $\pi_1(E)$ is trivial (cf. \cite{Kollar93, Takayama03}). However, this is not enough to conclude that $\pi_1(L(0\in X))$ is finite (e.g. consider a surface rational singularity that is not a quotient singularity).  

We can similarly define the local algebraic fundamental group $\hat{\pi}^{\loc}_1(X,0)$ which is just the pro-finite completion of $\pi_1^{\loc}(X,0)$. Our first theorem says that Koll\'ar's conjecture is true at least for  $\hat{\pi}^{\loc}_1(X,0)$.
\begin{thm}\label{t-local}
  Let $0\in (X,\Delta)$ be an algebraic klt singularity. Then the algebraic local fundamental group $\hat{\pi}^{\loc}_1(X,0)$ is finite. 
\end{thm}
We note that the main result of \cite{KK11} implies there exists an algebraic singularity $(X,0)$ with $\hat{\pi}_1(X,0)=\{e\}$ but $\pi_1(X,0)$ is an infinite group. 

The corresponding global result is the following.
\begin{thm}\label{t-global}
 Let $(X,\Delta+\Delta')$ be a projective klt pair with $\Delta'\ge 0$ and the coefficients of $\Delta$ are contained in $\{\frac{m-1}{m}|m\in \mathbb{N}\}$. Assume $-(K_X+\Delta+\Delta')$ is ample. 
 Denote by $X^0$  the maximal open locus where the restriction $(X,\Delta)$ is an orbifold. Then the algebraic orbifold fundamental group $\hat{\pi}^{\orb}_1(X^0,\Delta|_{X^0})$ is finite.  
\end{thm}

We note that the question on the fundamental group of the smooth locus of a log Fano variety has attracted lots of interests. When the dimension is equal to 2 and there is no boundary divisor, we know that the topological fundamental groups are always finite (see \cite{MT84, GZ95, KM99}) and we indeed have a classification of them (cf. \cite{Xu09}). 

Beyond only being an analog, we could really connect the above two theorems using the local-to-global induction, namely, we could show that the local result  Theorem \ref{t-local} implies the global result Theorem \ref{t-global} in the same dimension\footnote{After the note was posted on arXiv, this approach was worked out by Greb, Kebekus and Peternell in `Etale covers of Kawamata log terminal spaces and their smooth loci', arXiv:1302.1655.}, and then the global result gives the local one for one dimension higher.   However, to make the proof shorter, we  present the proof of Theorem \ref{t-global} using the boundedness theorem in \cite{HMX12}, and then establish Theorem \ref{t-local}.

\vspace{4mm}
\noindent{\bf  Acknowledgement}: We are in debt to J\'anos Koll\'ar for helpful discussions and to the anonymous referee for comments which improves the exposition. The author is partially supported by grant  DMS-1159175, Warnock Presidential Endowed Chair and the  grant 'Recruitment Program of Global Experts'.





\section{Finiteness of algebraic fundamental groups }

\noindent {\bf Notation and Convention:} We follow the terminology in \cite{KM98}. We call a finite morphism between two log pairs $f:(Y,\Delta_Y)\to (X,\Delta)$ {\it log \'etale in codimension }1 if $f^*(K_X+\Delta)=K_Y+{\Delta_Y}$. We note that here $\Delta$ and $\Delta_Y$ are effective divisors. 
A projective log pair $(X,\Delta)$ is called {\it log Fano} if $(X,\Delta)$ has klt singularities and $-(K_X+\Delta)$ is ample. Given a point $p$ on a log pair $(X,\Delta)$, we use $\mld(p,X,\Delta)$ to mean the  {\it minimal log discrepancy} $\min _E a(E,X,\Delta)+1$, where the minimum  runs over all exceptional divisors $E$ whose center on $X$ is $p$.

\vspace{8mm}
Let $f:Y\to X$ be a morphism induced by a surjection $\hat{\pi}^{\orb}_1(X^0,\Delta|_{X^0})\to G$ for some finite group $G$, then if we write $f^*(K_X+\Delta)=K_Y+\Delta_Y$, we have $\Delta_Y\ge 0$. Thus, if we write $f^*(K_X+\Delta+\Delta')=K_{Y}+\Delta_Y'$, it satisfies $\Delta_Y'\ge \Delta_Y \ge 0$. Therefore,  Theorem \ref{t-global} immediately  follows from the following result.

\begin{prop}\label{p-bound}
Let $(X,\Delta)$ be a log Fano variety. 
Let $f:Y\to (X,\Delta)$ be a finite surjective morphism, such that if we write 
$$f^*(K_X+\Delta)=K_{Y}+\Delta_Y,$$
then  $\Delta_Y$ is effective. Then the degree of $f$ is bounded by a constant $N$ only depending on $(X,\Delta)$. 
\end{prop}
\begin{proof}
 Let $M\in \mathbb{N}$ satisfy $M(K_X+\Delta)$ is Cartier, then $M(K_Y+\Delta_Y)$ is Cartier. It follows from \cite[Corollary 1.8]{HMX12} that 
$${\rm vol}(K_Y+\Delta_Y)=\deg(f)\cdot {\rm vol}(K_X+\Delta)$$
is bounded from above by a constant $C=C(M,n)$ which only depends on $M$ and $n=\dim (X)$. Thus $\deg(f)$ is bounded from above by 
$$\frac{C}{(-K_X-\Delta)^n}.$$
\end{proof}

As we mentioned in the introduction, here we use the strong boundedness result in \cite{HMX12}. We can argue more straightforwardly using Theorem \ref{t-local} in the same dimension. Later we will see, to prove Theorem \ref{t-local} we only need Proposition \ref{p-bound} in one dimension lower. 


The next lemma associates to every klt singularity an exceptional log canonical place after adding certain auxiliary divisor. Although the exceptional log canonical place depends on the choice of the auxiliary divisor, this construction proved to be useful for many questions (cf. \cite{Kollar07,HX09,LX11}).
\begin{lem}\label{l-plt}
Let $p\in (X,\Delta)$ be a klt point. There exists a $\mathbb{Q}$-divisor $H$ on $X$  and a birational morphism $f:Y\to X$ from a normal variety such that 
\begin{enumerate}
\item $Y$ has a prime divisor $E$ such that ${\rm Center}_X(E)=p$, $-(K_Y+f^{-1}_*\Delta+E)$ and $-E$ are ample over $X$ (in particular, ${\rm Ex}(f)={\rm Supp}(E)$), and
\item $(X,\Delta+H)$ is klt on $X\setminus \{p\}$, $\mld(p,X, \Delta+H)=0$ and $E$ is the unique divisor such that the discrepancy $a(E,X,\Delta+H)=-1$.
\end{enumerate}
\end{lem}
\begin{proof}
We first choose an ample $\mathbb{Q}$-divisor $L$ on $X$ such that $(X,\Delta+L)$ is log canonical at $p$ but klt at $X\setminus \{p\}$. Take a log resolution  $g:Z\to (X,\Delta+L)$  such that ${\rm Ex}(g)$ supports a fixed relative ample divisor $A$ over $X$. 
 We write $$g^*(K_X+\Delta+(t+\epsilon) L )\sim_{\mathbb{Q}} K_Z+g^{-1}_*(\Delta+tL)+\sum^k_{i=1} a_i E_i+(\epsilon g^*L+\delta A),$$
 where $0<\delta\ll \epsilon $ such that $\epsilon g^*L+\delta A\sim_{\mathbb{Q}} L'$ is a general ample $\mathbb{Q}$-divisor and each $a_i$ depends on $t$ and $\delta A$. Choosing $L$ a general ample $\mathbb{Q}$-divisor with small coefficients passing thought $p$, and using $\delta A$ to perturb, we can assume that there exists a $t_0>0$ such that  in the above formula, if we take $t=t_0$, there is a unique $a_i$, say $a_1$, which is equal to $1$, other $a_i<1$ $(i\ge 2)$ and the center of $E_Z:=E_1$ on $X$ is $p$.

 Considering the pair $(Z, g^{-1}_*(\Delta+tL)+E_Z+\sum^k_{i=2} E_i+L')$,
 we have
 $$K_Z+g^{-1}_*(\Delta+tL)+E_Z+\sum^k_{i=2} E_i+L'\sim_{X,\mathbb{Q}}\sum^m_{i=2}(1-a_i)E_i.$$
 
 We run an $(K_Z+g^{-1}_*(\Delta+tL)+E_Z+\sum^k_{i=2} E_i+L')$-MMP with scaling of $L'$  over $X$. By \cite{BCHM10}, this MMP will terminate with a good minimal model $h:W\to X$. As it contracts all the divisors whose supports are contained in the stable base locus, we know that $\phi:Z\dasharrow W$ precisely contracts all $E_i$ for $i\ge 2$.
 Thus  we have the divisorial part of ${\rm Ex}(h)$ is $E_W$, and on $W$ 
 $$K_W+h^{-1}_*(\Delta+tL)+E_W+\phi_*L'\sim_{X,\mathbb{Q}} 0,$$
 where $E_W$ denotes the push forward of $E_Z$ on $W$.
 Since $(W,h^{-1}_*(\Delta+tL)+E_W+\phi_*L')$ is dlt with only one divisor of coefficient 1, it is indeed plt. Furthermore, since it is an MMP with scaling of $L'$, by the definition of MMP with scaling (cf. \cite{BCHM10}), we know for some sufficiently small $\sigma>0$,
 $$K_W+h^{-1}_*(\Delta+tL)+E_W+(1+\sigma)\phi_*L' $$ is nef over $X$.
 We let $Y$ be the log canonical model of $(W, h^{-1}_*(\Delta+tL)+E_W+(1+\sigma)\phi_*L' )$, i.e.,
 $$Y={\rm Proj } \oplus_d h_*\mathcal{O}_W(d(K_W+h^{-1}_*(\Delta+tL)+E_W+(1+\sigma)\phi_*L' )).$$
As $\phi$ contracts $E_i$ $(i\ge 2)$, we know $\phi_*A\sim -\lambda E_W$ for some $\lambda>0$. Thus
$$K_W+h^{-1}_*(\Delta+tL)+E_W+(1+\sigma)\phi_*L' \sim_{X,\mathbb{Q}}\sigma\delta \phi_*A= -\sigma\delta\lambda E_W,$$
which is nef. 

Define $H=tL+(h\circ \phi)_*L'$. We see that $W\to Y$ can not contract $E_W$, thus it is a small morphism. As $(W, h^{-1}_*(\Delta+tL)+E_W+\phi_*L' )$ is plt, so is 
$(Y, E+f_*^{-1}(H+\Delta))$ where $E$ denotes the push forward of $E_W$ on $Y$ and we know $-E$ is $f$-ample.

Since $(X,\Delta)$ is klt, we know that
$$-(K_Y+f^{-1}_*\Delta+E)\sim_{X,\mathbb{Q}}- (1+a(E,X,\Delta))E$$
is $f$-ample.
 \end{proof}

\begin{rem}
In general, given a projective morphism $g:(X,\Delta)\to S$ from a klt pair to a normal varieties such that $g_*(\mathcal{O}_X)=\mathcal{O}_S$ and $-(K_{X}+\Delta)$ is ample over $S\setminus \{p\}$ for some point $p\in S$, then the same argument shows that we can find $f:Y\to S$ such that
\begin{enumerate}
\item $X$ and $Y$ are isomorphic over $S\setminus \{ p\}$,
\item  $f^{-1}(p)$ is an irreducible divisor $E$,
\item if we let $\Delta_Y$ be the birational transform of $\Delta$ on $Y$, then $-(K_Y+\Delta_Y+E)$ is $f$-ample. 
\end{enumerate}
\end{rem}

In this construction, we call $E$ a {\bf Koll\'ar component} of $(X,\Delta)$. As we mentioned, it depends on the auxiliary $\mathbb{Q}$-divisor $H$. If we write $(K_Y+f^{-1}_*\Delta+E)|_E=K_E+\Gamma$, where $\Gamma={\rm Diff}_{E}f_{*}^{-1}\Delta$ as defined in \cite[Section 16]{Kollaretal},
then the pair $(E,\Gamma)$ is log Fano. 

Koll\'ar component was first studied in \cite{Kollar07}. Later in \cite{LX11}, it was interpreted as the only remaining exceptional divisor after an MMP sequence scaled by a carefully chosen ample divisor as above to make it to be log Fano.

\begin{proof}[Proof of Theorem \ref{t-local}] Let $0\in (X,\Delta)$ be an  algebraic singularity on a pair $(X,\Delta)$.  Applying the construction in Lemma \ref{l-plt}, we denote by $f: Y\to X$ a morphism which precisely extracts a Koll\'ar component $E$. 

Now let 
$$\cdots\to (X_i,p_i)\to (X_{i-1},p_{i-1})\to \cdots\to (X_0,p_0)=(X,0),$$
be a sequence of finite morphisms such that each one is finite and \'etale for the restriction $X_{i+1}\setminus \{ p_{i+1}\} \to X_{i}\setminus \{p_{i}\}$ and Galois for $X_{i+1}\setminus \{ p_{i+1}\} \to X\setminus \{0\}$, we want to show it stabilizes for sufficiently large $i$.

For each $i$, we let $Y_i$ be the normalization of the main component of  $X_i\times _X Y$ with the morphism $f_i:Y_i\to X_i$.  Thus there are commutative diagrams
\begin{diagram}
E_{i+1}\subset Y_i & \rTo^{\psi_i} & E_{i}\subset Y_{i}\\
\dTo^{f_{i+1}} & & \dTo_{f_{i}}\\
p_{i+1}\in X_{{i+1}} & \rTo & p_{i}\in X_{i} & .
\end{diagram} 
Denote the pull back of $\Delta$ on $X$ by $\Delta_i$. Let $(K_{Y_i}+f_{i*}^{-1}\Delta_i+E_i)|_{E_i}=K_{E_i}+\Gamma_i$.
Since $\psi_i^*(f_{i*}^{-1}\Delta_{i})= f_{i+1*}^{-1}\Delta_{i+1}$, we conclude that 
$$\psi_i^*(K_{Y_{i}}+f_{i*}^{-1}\Delta_{i}+E_{i})=(K_{Y_{i+1}}+f_{i+1*}^{-1}\Delta_{i+1}+E_{i+1}).$$
Restricting on the Koll\'ar components, this implies that the induced morphism $$\psi_i|_{E_{i+1}}:(E_{i+1},\Gamma_{i+1})\to (E_{i},\Gamma_{i})$$ is  log \'etale in codimension 1. 

 It follows from Proposition \ref{p-bound} that there exists an $M\in \mathbb{N}$ such that $\psi_i|_{E_i}$ is an isomorphism for $i>M$. Fix such an $i>M$, thus $\psi_i$ is a finite morphism, totally ramified over $E_i$. Let $\gamma$ be the element in $\pi_1(L(p_i\in X_i))$ corresponding to the loop around a general point of $E_i$. We only need to verify the order of  $\gamma$ is finite. Cutting $Y_{i}$ to a surface $S_{i}$, and taking the Cartesian product, we have the following diagram 
\begin{diagram}
C_{i+1}\subset S_{i+1} & \rTo & C_{i}\subset S_{i}\\
\dTo^{\phi_{i+1}} & & \dTo_{\phi_{i}}\\
p_i\in T_{i+1} & \rTo & p_{i}\in T_{i}
\end{diagram} 

As the corresponding ramified covering is trivial along  $C_{i}={\rm Ex }(\phi_{i})$, we know that if we let the surjection $\hat{\pi}^{\loc}_1(T_{i},p_{i})\to G$ correspond to the covering, then $G$ is a finite cyclic group, which is generated by the image of $\gamma$. Thus ${\pi}^{\loc}_1(T_{i},p_{i})={\pi}_1(S_{i}\setminus C_{i})\to G$ indeed factors through $H_1(S_{i}\setminus C_{i})$. However, the homolog class $[\gamma]$ is in the kernel of  
$$H_1(S_{i}\setminus C_{i})\to H_1(S_{i})=H_1(C_{i}).$$
By Mumford's calculation on the normal surface singularity (cf. \cite[page 235]{Mumford61}), we know that $[\gamma]$ is a torsion element.  Thus for any $j\gg i$, $T_{j+1}\to T_{j}$ is an identity, and so is $X_{j+1}\to X_{j}$.

\end{proof}

Beijing International Center of Mathematics Research, 5 Yiheyuan
Road, Haidian District, Beijing, 100871, China

{\sl E-mail address: cyxu@math.pku.edu.cn}

\end{document}